\documentclass[final]{siamltex}
\usepackage{amsmath,amsopn,amssymb}
\allowdisplaybreaks

\numberwithin{equation}{section}
\newtheorem{Theorem}{Theorem}[section]

\newtheorem{Lemma}{Lemma}[section]
\newtheorem{Example}{Example}[section]
\newtheorem{Algorithm}{Algorithm}[section]
\catcode`@=11
\newbox\@temp
\def \blap#1{\vbox to 0mm{#1\vss}}

\newcommand{\bset}[3][2mm]{#3\llap{%
      \blap{\vskip#1\hbox to 0mm{\hss $#2$\hss}}%
      \setbox\@temp\hbox{\mbox{$#3$}}\hskip0.5\wd\@temp}}
\def\hhline{%
  \noalign{\ifnum0=`}\fi\hrule \@height 3\arrayrulewidth \futurelet
   \@tempa\@xhline}
\catcode`@=12

\title{The Shifting Technique for Computing the Extreme Solutions of
$X + A^\top X^{-1} A = Q$
}

\author{Chun-Yueh Chiang
\thanks{Center for General Education,National Formosa University, Huwei 632, Taiwan. {\tt (chiang@nfu.edu.tw)}}
\and Matthew M. Lin
\thanks{Corresponding Author. Department of Mathematics, National Chung Cheng University, Chia-yi 621, Taiwan. {\tt (mlin@math.ccu.edu.tw)}
}
}

\begin{document}
\maketitle
\begin{abstract}
We
propose a new way for speeding up the search of the maximal solution $X_+$ of $X + A^\top X^{-1} A = Q$.
It is known that the speed of convergence of traditional approaches for solving this problem depends highly on the spectral radius $\rho(X_+^{-1}A)$.  If $\rho(X_+^{-1}A)$ is close to one or equal to one, the iterations of traditional approaches converges very slowly or does not converge.  Our goal is to come up with a shifting tactic to remove the singularities embedded in $\rho(X_+^{-1}A)$. Finally, an example is used to demonstrate the capacity of our method.

\end{abstract}

\pagestyle{myheadings} \thispagestyle{plain}

%
%
%
\section{Introduction}
Consider the nonlinear matrix equation (NME)
\begin{eqnarray}
     X + A^\top X^{-1} A = Q, \label{eq:NMEs}
\end{eqnarray}
where  $A, Q \in \mathbb{R}^{n \times n}$ and $Q$ is  symmetric positive definite. We define two corresponding matrices of~\eqref{eq:NMEs}
\begin{equation}\label{ml}
   \mathcal{M} \equiv
   \left [
   \begin{array}{rc} A & 0 \\ Q & -I_n
   \end{array} \right ] , \quad
   \mathcal{L} \equiv \left [ \begin{array}{lc} 0 & I_n \\ A^\top & 0 \end{array}
   \right ].
\end{equation}
Note that the pencil $\mathcal{M} - \lambda
\mathcal{L}$ is symplectic, i.e., it satisfies
\[ \mathcal{M} J \mathcal{M}^\top = \mathcal{L} J \mathcal{L}^\top \ , \
\ \ \mathrm{with} \ \ J \equiv
\left [
\begin{array}{rc} 0 & I_n \\ -I_n & 0
\end{array} \right ],
\]
and $\lambda \in \sigma (\mathcal{M},\, \mathcal{L})$ if and only if
$1/\lambda \in \sigma (\mathcal{M},\, \mathcal{L})$. It is easy to see that
\begin{equation}\label{big}
\mathcal{M} \left [ \begin{array}{c} I_n \\ X \end{array}\right ] = \mathcal{L}
\left [ \begin{array}{c} I_n \\ X \end{array}\right ] X^{-1} A,
\end{equation}
and $\sigma(X^{-1} A)\subseteq\sigma(\mathcal{M} ,\mathcal{L} )$.

For a symmetric matrix $X$, we use the notation  $X \succ 0$ to say that $X$ is positive definite and the notation $X \succeq 0$ to say that $X$ is positive semidefinite. It follows that for two symmetric matrices $X$ and $Y$, we write  $X \succ Y$ if $X - Y \succ 0$ and $X \succeq Y$ if $ X-Y  \succeq 0$.

A symmetric solution $X_{+}$ of \eqref{eq:NMEs} is called maximal if $X_{+} \succeq X$ for any symmetric solution $X$ of \eqref{eq:NMEs}. Conditions for the existence of a symmetric
positive definite solution and a maximal symmetric positive definite
solution of~\eqref{eq:NMEs} are discussed in~\cite{EngRanRij_93}.
\begin{Theorem}~\label{thm20}
Let $\psi(\lambda)$ be a rational matrix-valued function defined by
\begin{equation}\label{eq:psi}
\psi(\lambda) = Q + \lambda A + \lambda^{-1} A^\top.
\end{equation}
Then,~\eqref{eq:NMEs} has a symmetric positive definite solution if and only if $\psi ( \lambda )$ is regular, i.e., $\det\psi (\lambda)\neq 0$ for some $\lambda\in\mathbb{C}$, and $\psi ( \lambda ) \geq 0$ for all $| \lambda | = 1$.
\end{Theorem}

Another necessary condition for the existence of a positive definite
solution for~\eqref{eq:NMEs} can be described as follows.
\begin{Theorem}\label{thm21}
If \eqref{eq:NMEs} has a symmetric positive
definite solution, then it has maximal
solution $X_+\succ 0$. 
Moreover, $X_+$ is the unique solution for which $\rho(X_+^{-1}A)\le 1$, that is, $\rho(X^{-1}A)> 1$ for any other solution $X\succ 0$. Here $\rho(\,\cdot\,)$
denote the spectral radius.
\end{Theorem}

The NME  arises in a large variety of disciplines in sciences and engineering. Its wide range of applications includes control theory, ladder networks, dynamic programming, stochastic filtering, and statistics (the reader is referred to~\cite{Anderson90,Zhan_96_SISC} for a list of references).
Many other aspects of NMEs, like solvability, numerical solution,
perturbation and applications, can be found in
\cite{Engw_93,EngRanRij_93,GuoLan_99_MC,Meini_02_MC,Sun03,Xu_00_ASNUP,Zhan_96_SISC,ZhanXie_96_LAA}
and the references therein.  Numerical approaches for obtaining the maximal solution of~\eqref{eq:NMEs} are mainly based on fixed-point iteration, Newton's iteration,
or a structure-preserving doubling algorithm (SDA)~\cite{Guo_01SIMAA,GuoLan_99_MC,Meini_02_MC, Lin06}. However, the convergence rates of all of these methods have been shown to be slow while $\rho(X^{-1}_+A ) \approx 1$. Note that given a matrix norm $\|\cdot\|$, an algorithm is linearly convergent if it generates a sequence of approximate solutions $\{X_k\}$ to the solution $X$ such that $\|X_k - X\|\leq \gamma \sigma^n$ for constants $0<\sigma<1$ and $\gamma > 0$ and quadratically convergent if $\|X_k - X\|\leq \gamma \sigma^{2^n}$.
Our major concern is to discuss the improvement of traditional approaches
while solving the NMEs with $\rho(X^{-1}_+A ) \approx 1$.

Our contribution in this work can be organized as follows. In Section 2, we review three iterative methods for finding the maximal solution of~\eqref{eq:NMEs}. In Section 3, we discuss how to shift eigenvalues of a general matrix pencil. A numerical example is given to demonstrate the application of the shifting technique.  In Section 4, we summarize our results and suggest an avenue for further research on applications of this shifting technique for solving matrix equations.

\section{Numerical approaches}
In this section we briefly review the numerical approaches
for solving NME and the corresponding convergence rate of each method.
We start our discussion with the fixed point iteration.

\subsection{Fixed point iteration}
Let $X_+$ be the maximal solution of~\eqref{eq:NMEs}.
\begin{Algorithm}\label{fix1}
{\emph{(Fixed point iteration for~\eqref{eq:NMEs})}}
\begin{enumerate}
\item  \emph{Take} $X_0 =Q$.
\item \emph{For} $k=0,1,\ldots,$ \emph{compute}
\begin{eqnarray}
X_{k+1}&=&Q-A^\top X_k ^{-1}A.\label{nmefix1}
\end{eqnarray}
\end{enumerate}
\end{Algorithm}

It has been shown in~\cite{GuoLan_99_MC} that the sequence  $\{X_k\}$ converges to $X_+$ and satisfies
\[
\limsup\limits_{k\rightarrow\infty} \sqrt[k]{\|X_k-X_+\|} \leq \rho(X_+^{-1}A)^2.
\]
To speed up computation, Zhan~\cite{Zhan_96_SISC} incorporated the Schulz iteration~\cite{Schulz33}
to provide an inversion-free variant method when $Q = I_n$.  This idea is then generalized in~\cite{GuoLan_99_MC} for general positive definite $Q$ as follows.
\begin{Algorithm}\label{fix2}
{\emph{(Inversion-free fixed point iteration for~\eqref{eq:NMEs})}}
\begin{enumerate}
\item  \emph{Take} $X_0 =Q$, $Y_0 = I_n/\|Q\|_\infty$.
\item \emph{For} $k=0,1,\ldots,$ \emph{compute}
\begin{subequations}\label{NMEINV}
\begin{eqnarray}
Y_{k+1}&=&Y_k(2I_n-X_kY_k),\label{NMEINV2}\\
X_{k+1}&=&Q-A^\top Y_k A.\label{NMEINV1}
\end{eqnarray}
\end{subequations}
\end{enumerate}
\end{Algorithm}
Note that Algorithm~\ref{fix2} requires more computation per iteration than Algorithm~\ref{fix1}. However, Its four matrix-matrix multiplication can be calculated in a parallel computing system effectively~\cite{GolVLoa1996}. Its numerical computation is more stable than Algorithm~\ref{fix1} due to the scheme for not computing the matrix inversions directly. The following result shows that the convergence rate of Algorithm~\ref{fix1} is roughly the same as that of Algorithm~\ref{fix2}.

\begin{Theorem}\cite{GuoLan_99_MC}\label{nmethmfix}
For $\epsilon >0$, the two generated sequences
$\{X_k\}$  and $\{Y_k\}$ of~\eqref{NMEINV} satisfy
\[
X_0\geq X_1\geq \cdots,\quad Y_0\leq Y_1\leq \cdots,
\]
\[
\|Y_{k+1} - X^{-1}_+\| \leq (\|AX_+^{-1}\| +\epsilon)^2 \|Y_k - X^{-1}_+\|,
\]
and
\[
\|X_k - X_+\| \leq \|A\|^2 \|Y_k - X^{-1}_+\|.
\]
\end{Theorem}

It follows that the convergence rate of iteration~\eqref{nmefix1} or \eqref{NMEINV} is highly related to the spectral radius $\rho(X_+^{-1}A)$.
Also, we see that Algorithm~\ref{fix1} and~\ref{fix2} are linearly convergent when $\rho(X_+^{-1}A) < 1$ and the convergences slow down when $\rho(X_+^{-1}A) $ is close to $1$. In these situation, an alternative approach, Newton's method, is recommended.
\subsection{Newton's method}
Let $\mathcal{P}^n$ be the set of positive definite matrices in $\mathbb{R}^{n\times n}$ and $\mathcal{S}^n$ be the set of all symmetric $n \times n$ real matrices. Corresponding to~\eqref{eq:NMEs}, we define an operator ${\mathcal R }:\mathcal{P}^n \rightarrow{\mathcal S^n } $ satisfying
\begin{align}\label{NMEOPERATOR}
{\mathcal R }(X)=-X+Q-A^\top X^{-1}A, \quad X\succ 0.
\end{align}
Then the Fr$\acute{e}$chet derivative
${\mathcal R }^{'}_X: \mathcal{S}^n \rightarrow \mathcal{S}^n$
of ${\mathcal R }$ at $X$ is given by
\begin{align}\label{Frechet}
{\mathcal R }^{'}_X(Z)=-Z+A^\top X^{-1}ZX^{-1}A, \quad Z\succ 0.
\end{align}
By~\eqref{NMEOPERATOR} and~\eqref{Frechet}, we obtain the Newton step for the solution of~\eqref{eq:NMEs} is
\[
X_k = X_{k-1} -  {(\mathcal{R}^{'}_{X_{k-1}})}^{-1}  {\mathcal R }(X_{k-1}), \quad k = 1,2, \ldots.
\]
Combining~\eqref{NMEOPERATOR} and~\eqref{Frechet}, we have the algorithm of Newton's method for~\eqref{eq:NMEs}.
\begin{Algorithm}\label{newton}
{\emph{(Newton's method for~\eqref{eq:NMEs})}}
\begin{enumerate}
\item  \emph{Take} $X_0 =Q$.
\item \emph{For} $k=0,1,\ldots,$ \emph{compute}
$L_k=X_{k-1}^{-1}A$ \emph{and solve}
\begin{eqnarray*}
X_k-L_k^\top X_k L_k=Q-2L_k^\top A.
\end{eqnarray*}
\end{enumerate}
\end{Algorithm}
It is known that in each iteration,  the computational work for Newton's method is about $15$ times larger than that
for the fixed point iteration. In order to have a better picture of the advantage of applying Newton's method, we include the convergence result discussed in~\cite{GuoLan_99_MC} as follows.
\begin{Theorem}\label{NMENTMTHM}
If \eqref{eq:NMEs} has a positive definite solution, then Algorithm~\ref{newton} determines a nondecreasing sequence of symmetric
matrices $\{X_k\}$ for which $\rho(L_k)<1$ and $\lim\limits_{k\rightarrow\infty} X_k=X_+$.
Moreover, if $\rho(X_+^{-1}A)<1$,
the convergence is quadratic and  if
$\rho(X_+^{-1}A)=1$ and all eigenvalues of $X_+^{-1}A$ on the unit  circle are semisimple, the convergence is either quadratic or
linear with rate $1/2$.
\end{Theorem}

Theorem~\ref{NMENTMTHM} states that  if
$\rho(X_+^{-1}A)=1$, convergence is guaranteed under a
addition assumption, i.e., all eigenvalues of $X_+^{-1}A$ on the unit circle are semisimple. In~\cite{Lin06}, Lin and Xu investigate the approach of applying the structure-preserving algorithm (SDA) for solving~\eqref{eq:NMEs} without any assumption on the
unimodular eigenvalues of $X_+^{-1}A$.

\subsection{SDA}
The fundamental idea of SDA is the doubling transformation. In this sense, we  begin with the discussion of the doubling transformation.

Let $ \mathcal{M} - \lambda  \mathcal{L}$ be a matrix pencil consisting of two matrices
 \begin{equation}\label{ssf}
 \mathcal{M} \:=
 \left [  \begin{array}{cr} A & 0 \\ Q & -I_n
\end{array} \right ]
\mbox{ and } \mathcal{L} \:= \left [ \begin{array}{cc} -P & I_n \\
A^\top & 0 \end{array}\right ],
\end{equation}
with $Q$, $P \succeq 0$. This is the so-called the \emph{second
standard symplectic form} (SSF-2). For any solution $\left [\mathcal{M}_{\star}, \mathcal{L}_{\star}\right ]$ within the left null space of $ \left[\begin{array}{c}\mathcal{L} \\-\mathcal{M}\end{array}\right]
$, define  $ \widehat{\mathcal{M}} :=
 \mathcal{M}_{\star} \mathcal{M}$ and $ \widehat{\mathcal{L}} :=
 \mathcal{L}_{\star} \mathcal{L}$. The transformation
 \begin{equation*}
 \mathcal{M} - \lambda  \mathcal{L} \rightarrow   \widehat{\mathcal{M}} - \lambda   \widehat{\mathcal{L}}
\end{equation*}
is called a \emph{doubling transformation}.
Assume further that
\begin{equation*}
\mathcal{M}_{\star} = \left [
\begin{array}{rc} A(Q-P)^{-1} & 0
\\ -A^\top (Q-P)^{-1} & I_n \end{array} \right ] \mbox{ and }
\mathcal{L}_{\star} :=
\left [
\begin{array}{cr}
I_n & -A(Q-P)^{-1} \\ 0 & A^\top (Q-P)^{-1} \end{array}
\right ].
\end{equation*}
We see that by direction computation,
 \begin{equation*}
 \widehat{\mathcal{M}} :=
 \mathcal{M}_{\star} \mathcal{M}
 =
 \left [
 \begin{array}{cr} \widehat{A}
& 0
\\ \widehat{Q} & -I_n \end{array} \right ] \mbox{ and }
\widehat{\mathcal{L}} :=
\mathcal{L}_{\star} \mathcal{L} =
\left [\begin{array}{rc} -\widehat{P} & I_n \\
\widehat{A}^\top & 0 \end{array}\right ],
\end{equation*}
where
\begin{equation}\label{aqp}
\widehat{A} := A(Q-P)^{-1} A,\quad
\widehat{Q} := Q - A^\top (Q-P)^{-1} A \mbox{ and }
\widehat{P} := P + A(Q-P)^{-1} A^\top.
\end{equation}
This implies that if $ Q - P \succ 0$ and $Q - A^\top(Q-P)^{-1} A \succeq 0$, then $(\widehat{\mathcal{M}},\, \widehat{\mathcal{L}})$ is again a SSF-2 form~\cite{Lin06}.
Based on formulae~\eqref{aqp}, we then have the following algorithm, SDA.
\begin{Algorithm}\label{SDA}
{\emph{(SDA for~\eqref{eq:NMEs})}}
\begin{enumerate}
\item  \emph{Take} $A_0 =A$, $Q_0 = Q$, $P_0 = 0$.
\item \emph{For} $k=0,1,\ldots,$ \emph{compute}
\begin{subequations}\label{NMESDA2}
\begin{eqnarray}
  A_{k+1}&=&A_k(Q_k-P_k)^{-1}A_k; \\
  Q_{k+1}&=&Q_k-A_k^\top(Q_k-P_k)^{-1}A_k; \\
  P_{k+1}&=&P_k+A_k(Q_k-P_k)^{-1}A_k^\top.
\end{eqnarray}
\end{subequations}
\end{enumerate}
\end{Algorithm}
Below we quote  from \cite[Theorem~2.1]{Lin06} to guarantee that Algorithm~\ref{SDA} is well-defined, that is, the difference $Q_k - P_k \succ 0$, for all $k$.
\begin{Theorem}\label{thm31}\cite{Lin06}
Let $X \succ 0$ be a solution of~\eqref{eq:NMEs}.
Define $S = X^{-1} A$. Then the sequences
$\{ A_k,\, Q_k,\, P_k \}$ generated by Algorithm~\ref{SDA} satisfy
\begin{enumerate}
\item $A_k = (X-P_k)S^{2^k}$;
\item $0 \leq P_k \leq P_{k+1} < X$ and $Q_k-P_k = (X-P_k) + A_k^\top
(X-P_k)^{-1}A_k > 0$;
\item $X \leq Q_{k+1} \leq Q_k \leq Q$ and $Q_k - X = (S^\top)^{2^k}
(X-P_k)S^{2^k} \leq (S^\top)^{2^k} XS^{2^k}$.
\end{enumerate}
Moreover, we have
\begin{enumerate}
\item $\|A_k\|_2\leq \|X\|_2\|S^{2^k}\|_2$;
\item $\|Q_k - X\|_2 \leq \|X\|_2 \|S^{2^k}\|_2^2$.
\end{enumerate}
\end{Theorem}

From Theorem~\ref{thm31}, we know that if $\rho{(X^{-1}A)} < 1$, then
the SDA is quadratically convergent. If $\rho{(X^{-1}A)} = 1$,
Chiang et al. in~\cite{ChiangPHD2009} proved that  Algorithm~\ref{SDA} for NME is linearly convergent with rate $1/2$, without any assumption on the unimodular eigenvalues of $\rho{(X^{-1}A)}$. In this case, we are interested in exploring a strategy of shifting unimodular eigenvalues of $X^{-1}A$ so that the convergence of above algorithms can be speeded up.

\section{The Shifting technique}
Assume that $X \succ 0$ is a solution of~\eqref{eq:NMEs}. Then the solution $X$ is highly related to the generalized eigenspace of the pencil
\begin{equation*}
\mathcal{M} - \lambda \mathcal{\mathcal{L}}
= \begin{bmatrix}
A   & 0%
\\ Q  &  -I_n
\end{bmatrix} - \lambda \begin{bmatrix}
0 & I_n%
\\ A^\top  & 0
\end{bmatrix}.
\end{equation*}
That is, if $X\succ 0$ is a solution of~\eqref{eq:NMEs}
if and only if $X$ satisfies~\eqref{big}.

Corresponding to~\eqref{big}, let us focus on the discussion with the
shifts of eigenvalues of the matrix pencil $\mathcal{M}-\lambda L$. To begin with, we consider a single shift of an eigenvalue of the matrix pencil $\mathcal{M}-\lambda L$.
\begin{Lemma}\label{lem1}
 Let $\mathcal{M}-\lambda L$ be a matrix pencil with $\mathcal{M}v = \lambda_0 Lv$ for some nonzero vector $v$. If $r$ is a vector with $r^{\top}v =1$, then for any scalar $\lambda_1$, the eigenvalues of the matrix pair
\begin{equation*}
\widehat{\mathcal{M}}-\lambda\widehat{L}  = \mathcal{M}+(\lambda_1-\lambda_0)Lvr^\top -\lambda L,
\end{equation*}
consist of those of $\mathcal{M}-\lambda L$, except that one eigenvalue $\lambda_0$ of $\mathcal{M}-\lambda L$  is replaced by $\lambda_1$.
\end{Lemma}
\begin{proof}
Since $\mathcal{M}v = \lambda_0 Lv$ and $r^{\top}v =1$, we have
\begin{align}\label{eq1}
\widehat{\mathcal{M}}v &=\lambda_1\widehat{L}v.
\end{align}
Note that  $(\mathcal{M} - \lambda \mathcal{L})v = \lambda_0 \mathcal{L}v -\lambda \mathcal{L}v = (\lambda_0 - \lambda)\mathcal{L}v$.
Also, for any $\lambda \neq \lambda_0$, we see that
\begin{eqnarray}
\det(\widehat{\mathcal{M}}-\lambda\widehat{\mathcal{L}})&=&\det(\mathcal{M}-\lambda \mathcal{L})\det(I_n+(\lambda_1-\lambda_0)Lvr^\top(\mathcal{M}-\lambda \mathcal{L})^{-1})\nonumber\\
&=&\det(\mathcal{M}-\lambda \mathcal{L})(1+(\lambda_1-\lambda_0)r^\top(\mathcal{M}-\lambda \mathcal{L})^{-1}\mathcal{L}v)\nonumber\\
&=&\frac{\lambda_1-\lambda}{\lambda_0-\lambda}\det(\mathcal{M}-\lambda \mathcal{L}).\label{eq2}
\end{eqnarray}
Thus, the theorem follows from~\eqref{eq1} and~\eqref{eq2}.

\end{proof}
Similar to the proof given above, we then have the following result that
$k$ eigenvalues of  $\mathcal{M}-\lambda \mathcal{L}$ are shifted simultaneously.
\begin{Theorem}\label{lem:shiftML r2 eigs}
Let $\mathcal{M}-\lambda \mathcal{L}$ be a matrix pencil with eigenvalues $\lambda_1,\cdots,\lambda_k$  that satisfy
\begin{equation*}
\mathcal{M}v_1 = \lambda_1 \mathcal{L}v_1 ,\cdots, \mathcal{M}v_k = \lambda_k \mathcal{L}v_k
\end{equation*}
for some nonzero vectors $v_1,\cdots,v_k\in \mathbb{C}^n$.
Suppose  $V= [v_1, \cdots ,v_k]\in\mathbb{C}^{n\times k}$, $\Lambda=\mbox{diag}(\lambda_1,\cdots,\lambda_k)$ and $\widehat{\Lambda}=\mbox{diag}(\hat{\lambda}_1,\cdots,\hat{\lambda}_k)\in \mathbb{C}^{k\times k}$ for some  $\hat{\lambda}_1,\cdots,\hat{\lambda}_k \in \mathbb{C}$.
If $R_1$ and $R_2$ are two matrices in $\mathbb{C}^{n\times k}$ such that
\begin{equation}\label{cond}
R_1^{\top}V =\widehat{\Lambda}-\Lambda \mbox{ and } R_2^{\top}V =0,
\end{equation}
then the eigenvalues of the matrix
\begin{equation*}
\widehat{\mathcal{M}}-\lambda \widehat{\mathcal{L}} : =  (M+\mathcal{L}VR_1^{\top})-\lambda (\mathcal{L}+\mathcal{M}VR_2^{\top}),
\end{equation*}
consist of those of $\mathcal{M}-\lambda \mathcal{L}$, except that eigenvalues $\lambda_1,\cdots,\lambda_k$
of $\mathcal{M}-\lambda \mathcal{L}$ are replaced by $\hat{\lambda}_1,\cdots,\hat{\lambda}_k$. Moreover,  $\hat{\lambda}_1,\cdots,\hat{\lambda}_k$ are eigenvalues of $\widehat{\mathcal{M}}-\lambda\widehat{\mathcal{L}}$ corresponding to the eigenvectors $v_1,\cdots,v_k$, respectively.
\end{Theorem}

\begin{proof}
Since $\mathcal{M}V = \mathcal{L}V\Lambda$, we have
$(\mathcal{M} - \lambda \mathcal{L})V 
= \mathcal{L}V(\Lambda - \lambda I_n)$. Thus,
$(\mathcal{M} - \lambda \mathcal{L})^{-1} \mathcal{L}V =
V(\Lambda - \lambda I_n)^{-1} $. This implies that for any $\lambda \neq \lambda_1,\cdots,\lambda_k$, the determinant of
$\widehat{\mathcal{M}} - \lambda \widehat{\mathcal{L}}$ is
\begin{eqnarray*}
\det(\widehat{\mathcal{M}} - \lambda \widehat{\mathcal{L}})
& = & \det(\mathcal{M} - \lambda \mathcal{L}) \det( I_n + (\mathcal{L}VR_1^{\top}
- \lambda \mathcal{L}V\Lambda R_2^{\top})
 (\mathcal{M}- \lambda \mathcal{L})^{-1})\\
& = & \det(\mathcal{M} - \lambda \mathcal{L}) \det( I_k + R_1^{\top}
(\mathcal{M} - \lambda \mathcal{L})^{-1} \mathcal{L}V -
\lambda\Lambda R_2^{\top}
(\mathcal{M} - \lambda \mathcal{L})^{-1} \mathcal{L}V
)\\
& = & \det(\mathcal{M} - \lambda \mathcal{L}) \det\left( I_k + (\widehat{\Lambda}-\Lambda)
\mbox{diag}\left(\frac{1}{\lambda_1 -\lambda}, \frac{1}{\lambda_2 -\lambda} \ldots, \frac{1}{\lambda_k -\lambda} \right )\right )\\
& = & \det(\mathcal{M} - \lambda \mathcal{L}) \left( \frac{\hat{\lambda}_1 -\lambda}{\lambda_1 -\lambda}\right )\left (\frac{\hat{\lambda}_2 - \lambda}{\lambda_2 -\lambda}\right ) \ldots \left ( \frac{\hat{\lambda}_k - \lambda}{\lambda_k -\lambda} \right ).
\end{eqnarray*}

Also, by~\eqref{cond}, we have
\begin{align*}
\widehat{\mathcal{M}} V
= \mathcal{L}V\widehat{\Lambda} =\widehat{\mathcal{L}} V\widehat{\Lambda}.
\end{align*}
This completes the proof.
\end{proof}

Note that the matrix $R_1$ can be obtained by using the Gram--Schmidt process to the column vectors of $V$ and $R_2$ can be obtained from the vectors in the orthogonal space of the space spanned by the column vectors of $V$. We use the following example to demonstrate an application of the shifting technique discussed above.

\begin{Example}
Assume $n = 1$. Then \eqref{eq:NMEs}  can be written as
\begin{eqnarray}
     x + \frac{a^2}{x} = q, \label{eq:NMEs1}
\end{eqnarray}
where  $a,q\in\mathbb{R}$ and $q>0$ and the corresponding matrix pencil is denoted by
\begin{align}\label{eq:m1}
\mathcal{M}_1 - \lambda \mathcal{L}_1
= \begin{bmatrix}
a   & 0%
\\ q  &  -1
\end{bmatrix} - \lambda \begin{bmatrix}
0 & 1%
\\ a  & 0
\end{bmatrix}.
\end{align}
Let $\Delta = q^2-4a^2$ be the discriminant of~\eqref{eq:NMEs1} and $\Gamma(\lambda) = \det\left(
\mathcal{M}_1 - \lambda \mathcal{L}_1\right)=-a\lambda^2+q\lambda-a$ be the determinant of~\eqref{eq:m1}. Corresponding to~\eqref{eq:psi}, we define the function $\psi_1(\lambda)$ such that
\begin{eqnarray}\label{eq:psi_1}
\psi_1(\lambda)&= \frac{a\lambda^{2}+q\lambda+a}{\lambda}.
\end{eqnarray}
It is clear that $\psi_1(\lambda)$ is regular.
Let $\lambda=e^{i\theta}$ for every  $\theta\in\mathbb{R}$. Substituting this $\lambda$ into~\eqref{eq:psi_1}, we obtain $\psi_1(e^{i\theta})=2a\cos(\theta)+q$.
This implies that $\psi_1(e^{i\theta})\geq 0$  for every  $\theta\in\mathbb{R}$
if and only if $\Delta \geq 0$.
Upon using Theorem~\ref{thm20}, we know that
\eqref{eq:NMEs1} has no symmetric positive definite solution, provided
$\Delta < 0$ and has a maximal solution $x_+ > 0$, provided $\Delta \geq 0$.  In particular, $x_+ \geq |a|$, since $\rho(x_+^{-1} a) < 1$ (see Theorem~\ref{thm21}).

%



Now we are ready to rewrite Algorithm~\eqref{SDA} corresponding to
~\eqref{eq:NMEs1} as follows.
\begin{Algorithm}\label{SDA1}
{\emph{(SDA for solving \eqref{eq:NMEs1})}}
\begin{enumerate}
\item  \emph{Take} $a_0 =a$, $q_0 = q$, $p_0 = 0$.
\item \emph{For} $k=0,1,\ldots,$ \emph{compute}
\begin{subequations}\label{NMESDA1}
\begin{eqnarray}
  a_{k+1}&=&\frac{a_k^2}{q_k-p_k}; \\
  q_{k+1}&=&q_k-a_{k+1}; \\
  p_{k+1}&=&p_k+a_{k+1}.
\end{eqnarray}
\end{subequations}
\end{enumerate}
\end{Algorithm}

 If the maximal solution  $x_+>0$ of~\eqref{eq:NMEs1} exists, by Theorem~\ref{thm31}, we have
\begin{eqnarray*}
a_k &= & (x_+-p_k)(\frac{a}{x_+})^{2^k}, \\
q_k-p_k &=& (x_+-p_k) + \frac{a_k^2}{x_+-p_k}, \\
q_k - x_+ &=& (x_+-p_k)(\frac{a}{x_+})^{2^{k+1}}.
\end{eqnarray*}

Observe further that if $x_+=|a|$,  we obtain $q=|a|+\frac{a^2}{|a|}=2|a|>0$
by substituting $x_+$ into~\eqref{eq:NMEs1}.  It follows that $\Delta = 0$ and any positive solution satisfies
\begin{align*}
(\sqrt{x}-\frac{a}{\sqrt{x}})^2=0.
\end{align*}
Without loss of generality, assume $a > 0$. 
%
By induction and~\eqref{NMESDA1}, it is easy to see that
\begin{align*}
&a_k=\frac{a}{2^k},\quad q_k=\frac{(2^k+1)a}{2^k},\quad q_k=\frac{(2^k-1)a}{2^k},\,\mbox{ for } k=0,1,\ldots\\
&\limsup\limits_{k\rightarrow\infty}\sqrt[k]{|q_k-a|}  =\frac{1}{2}.
\end{align*}
This means that the sequence $q_k$ converges linearly to $a$ with rate $\frac{1}{2}$. On the other hand, $\Gamma(\lambda)=-a(\lambda-1)^2$ implies that~\eqref{eq:m1} has the
the right eigenspace $E_1=\mbox{Ker}\{\mathcal{M}_1-1\mathcal{L}_1\}=\mbox{Span}\{\begin{bmatrix}1 \\ a\end{bmatrix}\} $
corresponding to the eigenvalue $1$. Let $v_1^\top = [1 , a]$. By Theorem~\ref{lem:shiftML r2 eigs}, define two corresponding vectors $r_1^\top = [r-1, 0]$, for some constant $0 < r <1$, and $r_2^{\top} = [0, 0]$ such that the new matrices $\widehat{\mathcal{M}}_1$ and  $ \widehat{\mathcal{L}}_1$ is given by
\begin{eqnarray*}
\widehat{\mathcal{M}}_1 &:=& \mathcal{M}_1 + \mathcal{L}_1v_1 r_1^\top =
\left[\begin{array}{cc} ar & 0 \\ a(1+r) & -1\end{array}\right],\\
\widehat{\mathcal{L}}_1 &:=& \mathcal{L}_1 + \mathcal{M}_1v_1 r_2^\top =
\left[\begin{array}{cc}0 & 1 \\a & 0\end{array}\right].\\
\end{eqnarray*}
By direct computation, we see that the eigenvalues of the new matrix pencil $\widehat{\mathcal{M}}_1 - \lambda \widehat{\mathcal{L}}_1$
 are $r$ and $1$ and the eigenvectors can be chosen as $[1,a]^\top$ and $[1, ar]^\top$, respectively. Let $\hat{v}_1^\top = [1 , ar]$. Using the same constant $r$ given above, we define two vectors $\hat{r}_1^\top = [\frac{1}{r}-1, 0]$ and $\hat{r}_2^{\top} = [0, 0]$ such that
 \begin{eqnarray*}
\widehat{\mathcal{M}}_2 &:=& \widehat{\mathcal{M}}_1 + \widehat{\mathcal{L}}_1\hat{v}_1 \hat{r}_1^\top =
\left[\begin{array}{cc} a & 0 \\ a(r + \frac{1}{r}) & -1\end{array}\right],\\
\widehat{\mathcal{L}}_2 &:=& \widehat{\mathcal{L}}_1 + \widehat{\mathcal{M}}_1\hat{v}_1 \hat{r}_2^\top =
\left[\begin{array}{cc}0 & 1 \\a & 0\end{array}\right].\\
\end{eqnarray*}
it is worth noting that structure the pencil $\widehat{\mathcal{M}}_2 - \lambda \widehat{\mathcal{L}}_2$  is indeed SSF-2. Also,
the eigenvalues of the matrix pencil  $\widehat{\mathcal{M}}_2 - \lambda \widehat{\mathcal{L}}_2$ are $r$ and $\frac{1}{r}$ with eigenvectors $[1, \frac{a}{r}]^\top$ and $[1, ar]^\top$.
 Note that the maximal solution $\hat{x}_+$ of
\begin{eqnarray}
     x + \frac{a^2}{x} = a(r + \frac{1}{r}), \label{eq:NMEs2}
\end{eqnarray}
 is $\hat{x}_+ = \frac{a}{r}$ and $\frac{a}{\hat{x}_+} = r < 1$.
This implies that if we apply Algortihm~\ref{SDA1} to find the maximal solution of~\eqref{eq:NMEs2}
the convergence rate of the iterations is quadratic.
By the continuity dependence of the solution of the NME,
the maximal solution of~\eqref{eq:NMEs1} can be obtained by
taking  $r\rightarrow 1^-$ so that $\hat{x}_+$ approaches $x_+$, the maximal solution of~\eqref{eq:NMEs1}.

\end{Example}

\section{Conclusion}
In this work, we provide an approache of shifting eigenvalues of general matrix equations. Our goal is to remove the singularities happened while solving NME. Currently, we have not applied this method to a much more general NME or any other nonlinear matrix equation. We believe this research would propose an avenue for speeding up the numerical approaches for solving nonlinear matrix equations.

  \section*{Acknowledgment}
This research work is partially supported by the National Science Council and
the National Center for Theoretical Sciences in Taiwan.


\end{document}